\documentclass[12pt]{article}   					
\usepackage{amsmath,amssymb,amsthm}
\usepackage{parskip}
\usepackage{tikz}
\newtheorem{lemma}{Lemma}
\newtheorem{thm}[lemma]{Theorem}

\title{Minimal tilings of a unit square}
\author{Iwan Praton\\ Franklin \& Marshall College\\ Lancaster, PA 17604}
\date{}
\begin{document}
\maketitle
\begin{abstract}
Tile the unit square with $n$ small squares. We determine
the minimum of the sum of the side lengths of the $n$ small squares,
where the minimum is taken over all tilings of the unit square
with $n$ squares.
\end{abstract}
There are many interesting questions that arise from placing
non-overlapping small squares inside a larger square. 
For example, if we require the small squares to have different
side lengths and to tile the large square, then we have
the classic ``squaring the square'' problem, popularized
in Martin Gardner's column (November 1958).
The website \texttt{squaring.net}
contains a trove of information about similar problems. 

Erd\H{o}s and Soifer introduced a different kind of question
in this situation: if we put $n$ non-overlapping small squares inside the
unit square, how big can the side lengths of 
small squares get? They defined a function $f_M(n)$
(essentially the largest possible sum of the side lengths
of the $n$ small squares) and gave a precise conjecture
of its value. Erd\H{o}s also offered a \$50 bounty
for its proof or disproof.
There has been some progress on this question
---see [2], [3], [4]---but the conjecture is still unsolved.

In this paper we look at the natural analogue $f_m$ of
the function $f_M$, where we investigate the minimum
instead of the maximum. More precisely, and to fix
our terminology and notation, let
 $n\neq 1,2,3,5$ be a positive integer and let $T$ be a 
 \emph{tiling} of the
unit square using $n$ small squares (thus the $n$ small squares
are placed inside the unit square, completely filling it,
and with their interiors non-overlapping). The small
squares are called tiles. Note that it is not possible
to tile the unit squares with $2,3$, or $5$ tiles;
we also exclude $n=1$ to avoid triviality. We define $\sigma(T)$
to be the total length of the tiling $T$: if $s_1,\ldots,s_n$ are
the side lengths of the tiles, then $\sigma(T)=s_1+\cdots
+s_n$. We want to find out just how \emph{small} $\sigma(T)$
can be; in order to make this question interesting, we'll insist
that each tile has positive length, so that $s_i>0$ for all
$1\leq i\leq n$.

Define the function $f_m(n)=\min_T \sigma(T)$,
where the minimum is taken over all possible tilings of the unit
square using $n$ tiles. The aim of this paper is to prove the 
following formula: for $k\geq 2$,
\[
f_m(n) =\begin{cases}%
3-\frac2{k} & \text{if $n=2k$;}\\
3-\frac1{k} & \text{if $n=2k+3$}.
\end{cases}
\] 

These values can be attained: if $n=2k$,
tile the unit square with one big tile of length $(k-1)/k$
and $2k-1$ small tiles of length
$1/k$ (there is essentially only one way to do this). If
$n=2k+3$, we start with the minimal tiling for
$2k$, then divide one of the tiles of length $1/k$ into 4 equal
tiles. The figures below show the cases $n=8$ and $n=11$
(i.e.,  $k=4$).

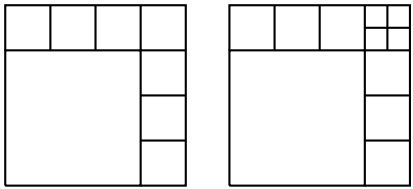
\begin{figure}[h]
\begin{center}
\begin{tikzpicture}[scale=0.6]
\draw[thick] (0,0) -- (3,0) -- (3,3) -- (0,3) -- (0,0);
\draw[thick] (3,0) -- (4,0) -- (4,4) -- (0,4) -- (0,3);
\draw[thick] (3,1) -- (4,1);
\draw[thick] (3,2) -- (4,2); \draw[thick] (3,4) -- (3,3) -- (4,3);
\draw[thick] (1,3) -- (1,4); \draw[thick] (2,3) -- (2,4); 
\end{tikzpicture}
\quad
\begin{tikzpicture}[scale=0.6]
\draw[thick] (0,0) -- (3,0) -- (3,3) -- (0,3) -- (0,0);
\draw[thick] (3,0) -- (4,0) -- (4,4) -- (0,4) -- (0,3);
\draw[thick] (3,1) -- (4,1);
\draw[thick] (3,2) -- (4,2); \draw[thick] (3,4) -- (3,3) -- (4,3);
\draw[thick] (1,3) -- (1,4); \draw[thick] (2,3) -- (2,4); 
\draw[thick] (3,3.5) -- (4,3.5);
\draw[thick] (3.5,3) -- (3.5,4);
\end{tikzpicture}
\end{center}
\end{figure}

(It turns out there is another
minimal tiling in the odd case. Start with three tiles of side
length $1/2$; then tile the remaining $1/2\times 1/2$ 
empty space with $2k$ tiles in the best (minimal) way.)

To start the proof, we introduce coordinates. Put
the unit square so that its corners are at $(0,0), (0,1),
(1,0)$, and $(1,1)$. For a given
tiling $T$, recall from [4]
the Staton-Tyler function $f$, where $f(c)$ is the number
of tiles that intersect the vertical line $x=c$  ($0\leq c\leq 1$).
Staton and Tyler showed that 
$\sigma(T)= \int_0^1 f(x)\,dx$. Note that the definition
of $f$ includes an ambiguity when the line $x=c$
intersects a vertical edge of a tile in $T$, but the
ambiguity is not harmful since we only use $f$
inside an integral.

Lemmas 1 to 4 below  have appeared
in [3] under a different context.

\begin{lemma} Suppose $T$ is a minimal tiling. Then
there exist two tiles whose lengths add up to $1$.
\end{lemma}
\begin{proof}
If not, then $f(x)\geq 3$ for all $x\in [0,1]$. Thus
$\sigma(T)=\int_0^1 f(x)\,dx\geq 3$, a contradiction.
\end{proof}

Let's say that $T$ is minimal and $A$ and $B$ are two tiles
with $s_A+s_B=1$.

\begin{lemma}
One (or both) of $A$ and $B$ must lie at a corner.
\end{lemma}
\begin{proof} Again by contradiction.

\begin{minipage}{0.7\textwidth}
If neither $A$ nor $B$ lie at a corner, then the tiling 
looks like the figure on the right.
The tiles on the left edge of the unit square, together
with the tiles on the right edge of the unit square,
have total length $2$. Note that neither $A$ nor $B$
are included in this count. Thus the total length of the
tiling is at least $3$, meaning $T$ is not minimal.
\end{minipage}
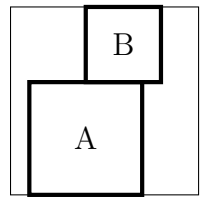
\qquad
\begin{minipage}{0.3\textwidth}
\begin{tikzpicture}[scale=0.25]
\draw (0,0) rectangle (10,10);
\draw[ultra thick] (1,0) rectangle (7,6);
\draw[ultra thick] (4,6) rectangle (8,10);
\node at (4,3) {A};
\node at (6,8){B};
\end{tikzpicture}
\end{minipage}
\end{proof}

We can actually conclude more than this. We can
assume without loss of generality that both $A$
and $B$ are corner tiles, as in the following
lemma.

\begin{lemma}
Suppose $A$ is a corner tile and $B$ is not. Then
there exists a similar tiling, with exactly the same
total length, where $B$ is a corner tile.
\end{lemma}
\begin{proof}
The proof is by picture:
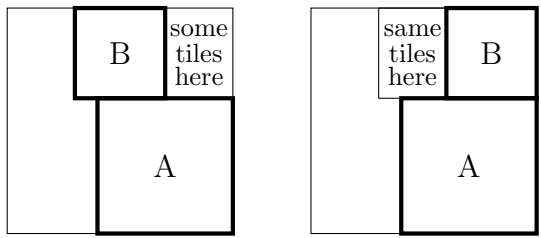
\begin{figure}[h]
\begin{center}
\begin{tikzpicture}[scale=0.3]
\draw (0,0) rectangle (10,10);
\draw[ultra thick] (4,0) rectangle (10,6);
\draw[ultra thick] (3,6) rectangle (7,10);
\node at (7,3) {A};
\node at (5,8){B};
\node at (8.5,9){\footnotesize some};\node at (8.5,8){\footnotesize tiles};
\node at (8.5,7){\footnotesize here};
\end{tikzpicture}
\qquad
\begin{tikzpicture}[scale=0.3]
\draw (0,0) rectangle (10,10);
\draw[ultra thick] (4,0) rectangle (10,6);
\draw[ultra thick] (6,6) rectangle (10,10);
\draw (3,6) rectangle (6,10);
\node at (7,3) {A};
\node at (8,8){B};
\node at (4.5,9){\footnotesize same};\node at (4.5,8){\footnotesize tiles};
\node at (4.5,7){\footnotesize here};
\end{tikzpicture}
\end{center}
\end{figure}

\end{proof}

We can thus assume that in our minimal tiling, one vertical
side of the unit square has just
two tiles, $A$ and $B$. We choose their names so that $s_A\geq s_B$.

\begin{lemma}
Let $A$ and $B$ be corner tiles as above, with
$s_A\geq s_B$. Then we can assume without harm
that there is another tile with length $s_B$ in the opposite
corner from $B$.
\end{lemma}
\begin{proof}
Turn the tiling with $A$ and $B$ by 90 degrees, then
apply
the same reasoning as before, especially lemmas 1
and 3. We get the conclusion.
\end{proof}

\begin{minipage}{0.6\textwidth}
We now have quite a bit of information about
the minimal tiling $T$. There is a large tile $A$
at one corner of the unit square, and on each
adjacent corner there is a tile of size $s_B=1-s_A$.
Note that this implies $s_A\geq 1/2$, so $A$
is the largest tile in $T$. 
\end{minipage}
\quad
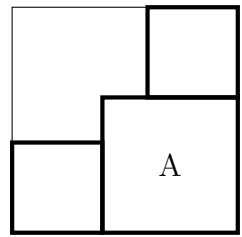
\begin{minipage}{0.4\textwidth}
\begin{center}
\begin{tikzpicture}[scale=0.3]
\draw (0,0) rectangle (10,10);
\draw[ultra thick] (4,0) rectangle (10,6);
\draw[ultra thick] (6,6) rectangle (10,10);
\draw [ultra thick] (0,0) rectangle (4,4);
\node at (7,3) {A};
\end{tikzpicture}
\end{center}
\end{minipage}

\medskip

We now concentrate on this tile $A$. It may be
the largest tile in the tiling, but it can't be \emph{too}
large. 

\begin{lemma}
Suppose $T$ is a tiling with $2k$ tiles. Let $A$
denote the largest tile in $T$, as above. Then 
$s_A\leq (k-1)/k$.
\end{lemma}
\begin{proof}
Every tile except $A$ has length at most $1-s_A$.
Therefore the total area of all tiles is at most
$(2k-1)(1-s_A)^2+s_A^2=2ks_A^2-(4k-2)s_A+(2k-1)$.
It is straightforward to verify that this quadratic in $s_A$
is smaller than 1 whenever $(k-1)/k <s_A<1$. Thus
we can only achieve a tiling when $s_A\leq (k-1)/k$.
\end{proof}

It would be nice to have a similar result when the number
of tiles in $T$ is odd, but the calculation is not as
straightforward. We begin with a careful consideration
of the placement of the tiles. 

\begin{lemma}
For $n\geq 2$, suppose 
$1/(k+1)<b<1/k$ and $T$ is a tiling with a large 
corner tile $A$ of length $s_A=1-b$.
Then $\sigma(T)\geq 3-b$.
\end{lemma}
\begin{proof}
As usual we put the unit square so it has corners
at $(0,0)$ and $(1,0)$. Place the tile $A$ so that
it has a corner at $(1,0)$. All tiles other than $A$
have length at most $b$. 

We now take a look at the top edge of the unit
square. There are at most $k$ tiles of length $b$
on this top edge (since $(k+1)b>1$). Suppose
there are indeed $k$ tiles of length $b$ on this
top edge. As in Lemma 3 we can assume that
these tiles are all on the right side, with no gaps
between them. Since $kb>1-b$, we have a situation
pictured on the left.

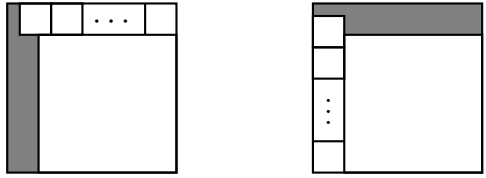
\begin{center}
\begin{tikzpicture}[scale=0.25]
\draw[fill=gray,gray] (0,0) rectangle (5/3,9-5/3);
\draw[fill=gray,gray] (0,22/3) rectangle (2/3,9);
\draw[thick] (0,0) rectangle (9,9);
\draw[thick] (5/3,0) rectangle (9,22/3);
\draw[thick] (7/3,22/3) rectangle (12/3,9);
\draw[thick] (2/3,22/3) rectangle (7/3,9);
\draw[thick] (22/3,22/3) -- (22/3,9);
\node at (17/3,24/3) {$\cdots$};
\end{tikzpicture}
\qquad\qquad
\begin{tikzpicture}[scale=0.25]
\draw[fill=gray,gray] (0,25/3) rectangle (5/3,9);
\draw[fill=gray,gray] (5/3,22/3) rectangle (9,9);
\draw[thick] (0,0) rectangle (9,9);
\draw[thick] (5/3,0) rectangle (9,22/3);
\draw[thick] (0,15/3) rectangle (5/3,20/3);
\draw[thick] (0,20/3) rectangle (5/3,25/3);
\draw[thick] (0,5/3) -- (5/3,5/3); 
\node at (5/6,11/3) {$\vdots$};
\end{tikzpicture}
\end{center}

If we now flip the tiling around the main diagonal of the unit square,
we get an equivalent tiling where there are at most
$(k-1)$ tiles of length $b$ on the top edge. Thus we
can harmlessly assume that the top edge of the unit
square has at most $(k-1)$ tiles of length $b$.
These tiles can further be assumed to be on the
right side, with no gaps between them, so they
all lie on the interval $(b,1)$.

Now recall the Staton-Tyler function $f$. 
 If $0< x< b$, then $f(x)\geq k+1$ since $k$ tiles,
each of length at most $b$, cannot add up to $1$. Now
suppose there are $m$ tiles of length $b$ in
the interval $b< x< 1$; these
tiles necessarily lie on top of tile $A$.
As noted above, we have $m\leq k-1$.
Note that in this interval, we have
$f(x)=2$ whenever we have a tile of size $b$, and $f(x)\geq 3$
otherwise. The $m$ tiles of size $b$ have total length
$mb$, so we have $f(x)\geq 3$ in an interval of length
$1-b-mb$.
Then
\begin{align*}
\sigma(T) &= \int_0^1 f(x)\,dx = \int_0^b f(x)\,dx +\int_b^1 f(x)\,dx\\
&\geq (k+1)b + 2mb + 3(1-b-mb)=3+(k-2-m)b\\
&\geq 3-b \quad \text{since $m\leq k-1$};
\end{align*}
this is exactly what we want.
\end{proof}

We can now get a result similar to Lemma 5 for odd
tilings.
\begin{lemma}
Suppose $T$ is a minimal tiling with $2k+3$ tiles.
As before, let $A$ denote the largest tile in $T$.
Then  $s_A\leq (k-1)/k$.
\end{lemma}
\begin{proof}
Suppose first that
$(k-1)/k < s_A < k/(k+1)$. Then $b=1-s_A$ satisfies
$1/k>b>1/(k+1)$. By the previous lemma, we have
$\sigma(T)\geq 3-b >3-1/k$, so $T$ is not minimal.
We can similarly rule out the case that $k/(k+1)<s_A
<(k+1)/(k+2)$. If $s_A\geq (k+1)/(k+2)$, then we need
at least $2(k+1)+1+1=2k+4$ tiles to tile the unit square,
which is too many. Thus we are reduced to the
case that $a=k/(k+1)$.

In this case, we redo the calculation that we did
in the $3-b$ lemma. The integral is exactly the 
same as before: if there are $m$ tiles ($0\leq m
\leq k-1$) of size
$b=1/(k+1)$ on the top of the unit square, then
$\sigma(T)\geq 3-b=3-1/(k+1)>3-1/k$, which
means $T$ is not minimal. Thus there must
be at least $k$ tiles on top of the unit square.
Similarly there are $k$ tiles on the left side of
of the unit square. This accounts for $2k+1$ tiles;
the empty space at the upper left corner is a
$b$-by-$b$ square, which cannot be tiled with
just two tiles. Thus this case also leads to $T$
being not minimal.
\end{proof}

Of particular interest is the case $k=2$.
In this case we conclude that $s_A$ is at most $1/2$,
but we know that $s_A\geq 1/2$. Thus $s_A=1/2$. By 
Lemma 4, we conclude
that $T$ has 3 tiles of length 1/2. The remaining
$1/2\times 1/2$ space must be tiled with four
tiles, and there is only one way to do that. So 
the length of the tiling is $5/2$, i.e., $f_m(7)=5/2$.

We are now ready to prove our main result.

\begin{thm}
For $k\geq 2$,
\[
f_m(n) =\begin{cases}%
3-\frac2{k} & \text{if $n=2k$},\\
3-\frac1{k} & \text{if $n=2k+3$}.
\end{cases}
\] 
\end{thm}
\begin{proof}
The proof is by induction on $k$. If $k=2$,
in the even case we have $n=4$, where
the result is clear; in the odd case we have $n=7$,
which we have already proved.  So assume
that $k>2$.

From Lemma 4 we know that it does no harm
to assume that our minimal tiling has a big tile $A$
located at a corner. For ease of reading write $a$ 
for the length $s_A$ of tile $A$. We can also 
assume that there are tiles of length $1-a$ 
adjacent to $A$. 

\begin{minipage}{3.8in}
Suppose first that $T$ contains $2k$ tiles.
Suppose further that $a>1/2$.
Notice that there is a tiling of the upper left $a\times a$
subsquare. In the picture
this subsquare is indicated by the dashed line.
\end{minipage}
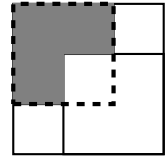
\qquad
\begin{minipage}{1.5in}
\begin{tikzpicture}
\draw[thick] (0,0) rectangle (2,2);
\draw[thick] (2/3,0) rectangle (2,4/3);
\draw[thick] (0,2/3) -- (2/3,2/3); \draw[thick] (4/3,4/3) -- (4/3,2);
\draw[fill=gray,gray] (0,2/3) rectangle (2/3,2);
\draw[fill=gray,gray] (2/3,4/3) rectangle (4/3,2);
\draw[dashed,ultra thick] (0,2/3) rectangle (4/3,2);
\end{tikzpicture}
\end{minipage}

The number
of tiles in this tiling of the subsquare is $2k-2=2(k-1)$.
Applying the induction hypothesis, the total length of
this tiling is at least $a(3-2/(k-1))$. Therefore the total
length of the original tiling of the unit square is
at least
\[
a\left(3-\frac2{k-1}\right) +2(1-a)+a -[a-(1-a)]=3-\frac{2a}{k-1}.
\]
By Lemma 5 $a\leq (k-1)/k$, so
\[
3-\frac{2a}{k-1}\geq 3-\frac{2(k-1)/k}{k-1} = 3-\frac{2}{k},
\]
as required.

There is also the pesky possibility that $a=1/2$. In this
case the unit square is divided into four subsquares,
where the upper left $1/2\times 1/2$ subsquare is
further tiled into $2k-3=2(k-3)+3$ tiles. This might
not be possible---indeed, if $k=3,4$, then this
subsquare cannot be tiled with 3 or 5 tiles. In such
cases we are finished. So we can assume that
$k\geq 5$. By the induction hypothesis,
the total length of this subtiling is at least 
$(1/2)(3-1/(k-3))=3/2-1/(2k-6)$. Thus the total length
of the tiling of the unit square is $3/2-1/(2k-6)+3/2
=3-1/(2k-6)$. This is bigger than $3-2/k$ for $k\geq 5$,
and we are done.

We now have to deal with the case where
$T$ has $2k+3$ tiles. 
As before, we first consider the case
where there is a big tile of size $a>1/2$; the picture
is as above. The dashed subsquare is tiled with
$2k+1$ tiles. By the induction hypothesis, the
length of the tiling is at least $a(3-1/(k-1))=
3a-a/(k-1)$. Thus the total length of the tiling
of the unit square is at least
\[
3a-\frac{a}{n-1}+ 2(1-a) + a - [a-(1-a)]=3-\frac{a}{k-1}.
\]
We know from Lemma 7 that $a\leq
(k-1)/k$, so 
\[
3-\frac{a}{k-1}\geq 3-\frac{(k-1)/k}{k-1} = 3-\frac1{k},
\]
as required.

We also need to consider the case where $a=1/2$.
In this case, the unit square is divided into four
$1/2\times 1/2$ subsquares, and one of these
subsquares is tiled into $2k$ tiles. We already
showed above that the length of this tiling is at least
$(1/2)(3-2/k)=3/2-1/k$, so the total length
of the tiling of the unit square is at least
$3/2-1/k+3/2=3-1/k$, which is what we require.

\end{proof}

\end{document}